\documentclass[11pt,a4paper]{article}

\title{\bf Stochastic Lagrangian flows on the group of volume-preserving
homeomorphisms of the spheres}
\author{Dejun Luo\footnote{Email: luodj@amss.ac.cn. Supported in part by the Key Laboratory of
RCSDS, CAS (2008DP173182), NSFC (11101407, 11371099) and AMSS (Y129161ZZ1)}
\vspace{3mm}\\
{\footnotesize Institute of Applied Mathematics, Academy of Mathematics and Systems Science,}\\
{\footnotesize Chinese Academy of Sciences, Beijing 100190, China}
}
\date{}

\usepackage{amssymb,amsmath,amsfonts,amsthm,color}

\def\R{\mathbb{R}}
\def\E{\mathbb{E}}
\def\T{\mathbb{T}}
\def\d{\textup{d}}
\def\<{\langle}
\def\>{\rangle}
\newcommand{\ra}{\rightarrow}

\def\F{\mathcal{F}}
\def\D{\mathcal{D}}
\def\G{\mathcal{G}}

\def\L{\mathcal{L}}

\def\div{\textup{div}}

\newtheorem{theorem}{Theorem}[section]
\newtheorem{lemma}[theorem]{Lemma}       
\newtheorem{corollary}[theorem]{Corollary}
\newtheorem{proposition}[theorem]{Proposition}
\newtheorem{remark}[theorem]{Remark}

\begin{document}

\maketitle
\makeatletter 
\renewcommand\theequation{\thesection.\arabic{equation}}
\@addtoreset{equation}{section}
\makeatother 

\vspace{-6mm}

\begin{abstract}
We consider stochastic differential equations on the group of volume-preserving
homeomorphisms of the sphere $S^d\,(d\geq 2)$. The diffusion part is given by the
divergence free eigenvector fields of the Laplacian acting on $L^2$-vector fields,
while the drift is some other divergence free vector field. We show that the equation
generates a unique flow of measure-preserving homeomorphisms when the drift has first
order Sobolev regularity, and derive a formula for the distance between two Lagrangian flows.
We also compute the rotation process of two particles on the sphere $S^2$
when they are close to each other.
\end{abstract}



\section{Introduction}

It is well known that the Euler equation in hydrodynamics
  \begin{equation}\label{Euler}
  \frac{\partial u}{\partial t}+(u\cdot\nabla)u=\nabla p,\quad \div(u)=0
  \end{equation}
describes the evolution of the velocity $u$ of the non-viscous fluid.
In 1966, Arnold \cite{Arnold66} gave a geometric interpretation to \eqref{Euler}.
More precisely, he found that the flow $g(t)$ of volume-preserving
homeomorphisms on a manifold $M$ is a critical point of the energy
functional
  $$S[g]=\frac12\int_0^T\|\dot g(t)\|_{L^2}^2\,\d t,$$
if and only if its velocity $u(t,x)=\dot g(t,g^{-1}(t,x))$ solves the Euler
equation \eqref{Euler}. Inspired by this pioneer work, many
researchers have tried to establish variational formulations for
the Navier--Stokes equation
  \begin{equation}\label{Navier-Stokes}
  \frac{\partial u}{\partial t}+(u\cdot\nabla)u-\nu\triangle u=\nabla p,\quad \div(u)=0,
  \end{equation}
where $\nu>0$ is the viscosity of the fluid.

Following some ideas in \cite{Nakagomi, Yasue},
Cipriano and Cruzeiro \cite{CiprianoCruzeiro07}
presented a stochastic variational principle for the Navier--Stokes equation
on the two dimensional flat torus $\T^2$, with a modified energy functional
$S$. Their approach is based on the construction of a diffusion
process on the group of measure-preserving homeomorphisms on the
torus, and the generator of this diffusion coincides with
the Laplacian operator of $\T^2$. Malliavin \cite{Malliavin99} first
constructed such a flow on the group of homeomorphisms of the unit circle
(cf. \cite{AiraultRen02, Fang02} for more detailed studies).
After the work \cite{CiprianoCruzeiro07}, a series of papers
by Cruzeiro and her collaborators appeared, see for instance
\cite{CruzeiroFlandoliMalliavin07, AC12a, AC12b, AAC13}. In particular,
Arnaudon and Cruzeiro \cite{AC12b} extended the stochastic variational principle
to the Navier--Stokes equation on a general compact Riemannian manifold
$M$ without boundary, and studied the stability of flows on the group of
homeomorphisms of $M$.

In this paper we consider the stochastic Lagrangian flow on the group of
volume-preserving homeomorphisms on the sphere $S^d\, (d\geq 2)$, with the purpose of giving
another example for the general framework studied in \cite{AC12b, AAC13}. Our motivation
comes also from the studies in \cite{LeJan} of Sobolev isotropic flows on $S^d$ and $\R^d$
which serve as examples for illustrating the notion of statistical solution
proposed in that paper. Later on, it was shown in \cite{FangZhang06} that the isotropic flow
on $S^d$ corresponding to the critical Sobolev exponent is indeed a flow of homeomorphisms
(see \cite{Luo} for the case of $\R^d$). The large deviation principles of the isotropic flow
of homeomorphisms on $S^d$ are studied in \cite{XuZhang, Wang}.

Inspired by these works, the divergence
free eigenvector fields of the Laplace operator acting on vector fields over
$S^d$, which constitute an orthonormal basis of the space of square integrable and
divergence free vector fields, will be taken as the diffusion coefficients in this paper.
In Section \ref{sect-notations}, we present some useful properties of these vector
fields and some notations concerning the group of volume-preserving homeomorphisms of $S^d$.
Then we prove in Section \ref{sect-SLF} the existence of a unique flow associated to stochastic
equations on the group of homeomorphisms when the drift is a divergence free vector
field belonging to $H^1$. We derive in Section \ref{sect-distance} a formula for the distance
between two stochastic Lagrangian flows generated by the same equation but with different initial
conditions. Finally we confine ourselves to $S^2$ in Section \ref{sect-rotation} and
compute explicitly the rotation of two particles when their distance is small,
by applying the formula given in \cite[Lemma 7.1]{AC12b}.

\section{Notations and preliminary results} \label{sect-notations}

We first introduce some notations necessary for defining the measure-preserving isotropic
flow on the sphere $S^d$ (see also \cite[Section 9]{LeJan} or \cite[Section 1]{FangZhang06}
for the more general case where the flows may be compressible).
Let $d(\cdot,\cdot)$ be the Riemannian distance function on $S^d$ which satisfies
  $$\cos d(x,y)=\<x,y\>_{\R^{d+1}},\quad x,y\in S^d,$$
thus $d(x,y)$ is exactly the angle between $x$ and $y$.
We have for all $x,y\in S^d$,
  \begin{equation}\label{sect-2.0}
  |x-y|_{\R^{d+1}}\leq d(x,y)\leq \frac\pi 2 |x-y|_{\R^{d+1}}.
  \end{equation}

Let $\Delta$ be the Laplacian operator acting on vector fields over
$S^d$. It is well known that the space of vector fields is the
direct sum of the subspace of gradient vector fields and that of
divergence free vector fields. Since we are concerned with the flow of
measure-preserving homeomorphisms on $S^d$, we only need the eigenvector fields of
$\Delta$ which are divergence free.

For $\ell\geq 1$, set $c_{\ell,\delta}=(\ell+1)(\ell+d-2)$. Then $\{c_{\ell,\delta};\ell\geq 1\}$
are the eigenvalues of $\Delta$ corresponding to the divergence free eigenvector fields.
Denote by ${\cal D}_\ell$ the eigenspace associated to $c_{\ell,\delta}$ and
$D_\ell=\hbox{\rm dim}({\cal D}_\ell)$ the dimension of ${\cal D}_\ell$.
It is known that
  $$D_\ell\sim O(\ell^{d-1})\quad \hbox{\rm as}\ \ \ell\ra +\infty.$$
For $\ell\geq1$, let $\{A_{\ell,k}; k=1, \ldots, D_\ell\}$ be an orthonormal basis of
${\cal D}_\ell$ in $L^2$:
  $$\int_{S^d} \big<A_{\ell,k}(x), A_{\alpha,\beta}(x) \big>\, \d x
  =\delta_{\ell\alpha}\delta_{k\beta}.$$
Weyl's theorem implies that the vector fields $\{A_{\ell,k};k=1,\ldots,D_\ell,\, \ell\geq1\}$ are smooth.

For $s>0$, let $G^s(S^d)$ be the infinite-dimensional group of homeomorphisms of $S^d$
which belong to the Sobolev space $H^s$ of order $s$. Denote by $G_V^s(S^d)$ the subgroup
consisting of volume-preserving homeomorphisms. In the following we shall simply
write $G^s$ and $G_V^s$ whenever there is no confusion. The Lie algebras of $G^s$
and $G_V^s$ are denoted by $\G^s$ and $\G_V^s$, respectively.
$\G^s$ is a Sobolev space of vector fields on $S^d$,
which is the completion of smooth vector fields with respect to the norm
  $$\|V\|_{H^s}^2 =\int_{S^d}\bigl<(-\Delta +1)^sV,V \bigr>\, \d x.$$
As a subspace of $\G^s$, $\G_V^s$ consists of divergence free vector fields,
having $\big\{ A_{\ell,k}/ (1+c_{\ell,\delta})^{s/2};k=1,\ldots,D_\ell,\ell\geq1\big\}$
as an orthonormal basis.

Now we collect some useful identities regarding the eigenvector fields
$\{A_{\ell,k}; k=1,\ldots,D_\ell,\, \ell\geq1\}$. To this end, denote by $c_d=\int_0^\pi
\sin^d\varphi\,\d\varphi$ and for $\theta\in[0,\pi]$,
  $$\gamma_\ell(\cos\theta)=\frac1{c_d}\int_0^\pi \big(\cos\theta-\sqrt{-1}\,\sin\theta\cos\varphi\big)^{\ell-1}
  \sin^d \varphi\,\d\varphi.$$
Remark that $\gamma_\ell$ is a real-valued function and $\gamma_\ell(1)=1$, $|\gamma_\ell(t)|\leq 1$ for all $t\in[-1,1]$.

\begin{lemma}\label{sect-2-lem-1}
\begin{itemize}
\item[\rm(i)] Let $\nabla$ be the covariant derivative on $S^d$. We have
  $$\sum_{k=1}^{D_\ell} \nabla_{A_{\ell,k}} A_{\ell,k}=0.$$
\item[\rm(ii)] Let $x,y\in S^d$ and $\theta$ the angle between $x$ and $y$. Then
  $$\frac{d}{D_\ell}\sum_{k=1}^{D_\ell}\<A_{\ell,k}(x),A_{\ell,k}(y)\>_{\R^{d+1}}
  =d\cos\theta\, \gamma_\ell(\cos\theta)-\sin^2\theta\, \gamma'_\ell(\cos\theta).$$
\item[\rm(iii)] Let $x,y\in S^d$ and $\theta$ the angle between them. Then
  \begin{align*}
  \frac{d}{D_\ell}\sum_{k=1}^{D_\ell}\<A_{\ell,k}(x),y\>_{\R^{d+1}}^2&=\sin^2\theta,\cr
  \frac{d}{D_\ell}\sum_{k=1}^{D_\ell}\big(\<A_{\ell,k}(x),y\>_{\R^{d+1}}+\<A_{\ell,k}(y),x\>_{\R^{d+1}}\big)^2
  &=2\sin^2\theta [1-\gamma_\ell(\cos\theta)].
  \end{align*}
\end{itemize}
\end{lemma}

\begin{proof}
These equalities are taken from \cite[Propositions A.3--A.5]{FangZhang06}.
Here we omit their proofs to save space.
\end{proof}

Let $\{w_{\ell,k}(t);\ 1\leq k\leq D_\ell,\ \ell\geq 1\}$
be a family of real independent standard Brownian
motions defined on a probability space $(\Omega, \F,P)$. Define
  $$W(t,x)=\sum_{\ell\geq 1} \sqrt{\frac{db_\ell}{D_\ell}}
  \sum_{k=1}^{D_\ell}w_{\ell,k}(t)A_{\ell,k}(x),$$
where $\{b_\ell\}_{\ell\geq1}$ is a family of nonnegative constants satisfying $\sum_{\ell\geq1}
b_\ell<+\infty$. Then $W$ is an isotropic Gaussian vector field which is divergence free with the
covariance function $C$ given by (see \cite[pp. 852--853]{LeJan})
  \begin{equation}\label{covariance}
  C\big((x,u),(y,v)\big)=\phi(\cos\theta)\<u,v\>+\psi(\cos\theta)\<y,u\>\<x,v\>,
  \end{equation}
where $x,y\in S^d$, $u\in T_x S^d,v\in T_y S^d$ and $\cos\theta=\<x,y\>$. The functions $\phi$ and
$\psi$ are defined as
  $$\phi(t)=\sum_{\ell\geq1}b_\ell\bigg(t\gamma_\ell(t)-\frac{1-t^2}{d-1}\gamma'_\ell(t)\bigg),\quad
  \psi(t)=\sum_{\ell\geq1}b_\ell\bigg(-\gamma_\ell(t)-\frac{t}{d-1}\gamma'_\ell(t)\bigg).$$
For example, if $b_1=0$ and
  \begin{equation}\label{sect-2.1}
  b_\ell=\frac{b}{(\ell-1)^{1+\alpha}},\quad \ell\geq 2,
  \end{equation}
where $\alpha>0,\,b>0$ are two constants, then we have
  \begin{equation*}
  \sqrt{\frac{b_\ell}{D_\ell}}\sim O\bigg(\frac1{\ell^{(\alpha+d)/2}}\bigg),
  \end{equation*}
and $\{W(t)\}_{t\geq0}$ is a cylindrical Brownian motion in the Sobolev space $\G_V^{(\alpha+d)/2}$
of divergence free vector fields. The Sobolev embedding theorem asserts that if $\alpha>2$,
then $W(t)$ takes values in the space of $C^1$-vector fields. In this case, we can apply
Kunita's classical method \cite{Kunita90} to conclude that the following SDE
  \begin{equation}\label{sect-2.3}
  \d X_t= \sum_{\ell\geq 1}\sqrt{\frac{db_\ell}{D_\ell}}
  \sum_{k=1}^{D_\ell}A_{\ell,k}(X_t)\circ \d w_{\ell,k}(t)
  \end{equation}
generates a stochastic flow
of diffeomorphisms on $S^d$. In the critical case, i.e., $\alpha=2$, S. Fang and T. Zhang
\cite{FangZhang06} proved that \eqref{sect-2.3}
determines a flow of homeomorphisms on $S^d$. Since the vector fields
$A_{\ell,k}$ are divergence free, we see that $X_t$ preserves the volume measure
of the sphere $S^d$. Moreover, it was shown in \cite[Theorem 6.1]{FangLuo07} that
if $\theta_0\in C(S^d)$, then $\theta(t,x):=\theta_0\big(X^{-1}_t(x)\big)$ solves
the corresponding stochastic transport equation in the distributional sense.
This fact is closely related to the notion of generalized flows proposed in
\cite[Definition 2.7]{AAC13}.

We introduce the following notation which will be used later: define
  \begin{equation}\label{G-theta}
  G(\theta)=\sum_{\ell=1}^\infty b_\ell\, \gamma_{\ell}(\cos\theta),\quad \theta\in [0,\pi].
  \end{equation}
Then we have
  $$\phi(\cos\theta)=\cos\theta\, G(\theta)+\frac{\sin\theta}{d-1}G'(\theta),\quad
  \psi(\cos\theta)=-G(\theta)+\frac{\cot\theta}{d-1}G'(\theta).$$
In the critical case, i.e., $b_\ell$ is defined as in \eqref{sect-2.1} with $\alpha=2$,
it was shown in \cite[Proposition 2.2]{FangZhang06} that there exists some $C>0$ such that
  \begin{equation*}
  |G'(\theta)|\leq C\theta \log \frac{2\pi}\theta,\quad \theta\in[0,\pi].
  \end{equation*}
As mentioned at the beginning of \cite[Section 4]{AC12b}, there is no canonical choice
of the Brownian motion $W(t)$ in the space $\G_V^0$ of divergence free vector fields.
Throughout this paper, we shall fix a sequence $\{b_\ell\}_{\ell\geq 1}$ such that the function
$G\in C^2([0,\pi])$ and
  \begin{equation}\label{G-theta.1}
  |G'(\theta)|\leq C\theta,\quad \theta\in[0,\pi].
  \end{equation}
For example, this is the case if $b_\ell$ is chosen as in \eqref{sect-2.1} with $\alpha>2$.
The following result is a simple consequence of Lemma \ref{sect-2-lem-1}.

\begin{corollary}\label{sect-2-cor}
Let $\theta$ be the angle between $x,y\in S^d$. Then
  \begin{align*}
  \sum_{\ell=1}^\infty \frac{db_\ell}{D_\ell}\sum_{k=1}^{D_\ell}
  |A_{\ell,k}(x)-A_{\ell,k}(y)|_{\R^{d+1}}^2
  &=2d \big[G(0)-\cos\theta\, G(\theta)\big]-2\sin\theta\, G'(\theta),\cr
  \sum_{\ell=1}^\infty \frac{db_\ell}{D_\ell}\sum_{k=1}^{D_\ell}
  \big\<x-y,A_{\ell,k}(x)-A_{\ell,k}(y)\big\>_{\R^{d+1}}^2
  &=2\sin^2\theta \big[G(0)-G(\theta)\big].
  \end{align*}
\end{corollary}

\begin{proof}  We simply write $|\cdot|$ and $\<\cdot,\cdot\>$ for the Euclidean norm
and inner product in $\R^{d+1}$. By Lemma \ref{sect-2-lem-1}(ii), we have
  \begin{align*}
  \sum_{\ell=1}^\infty \frac{db_\ell}{D_\ell}\sum_{k=1}^{D_\ell}
  |A_{\ell,k}(x)-A_{\ell,k}(y)|^2
  &=\sum_{\ell=1}^\infty \frac{db_\ell}{D_\ell}\sum_{k=1}^{D_\ell}
  \big(|A_{\ell,k}(x)|^2+|A_{\ell,k}(y)|^2-2\<A_{\ell,k}(x),A_{\ell,k}(y)\>\big)\cr
  &=\sum_{\ell=1}^\infty b_\ell \big[2d\big(1-\cos\theta\, \gamma_\ell(\cos\theta)\big)
  +2\sin^2\theta\, \gamma'_\ell(\cos\theta)\big].
  \end{align*}
Then the first identity follows from the definition of $G(\theta)$. Next,
$\<A_{\ell,k}(x),x\>=0$ for any $x\in S^d$, therefore
  \begin{align*}
  \big\<x-y,A_{\ell,k}(x)-A_{\ell,k}(y)\big\>
  =-\<x, A_{\ell,k}(y)\>-\<y,A_{\ell,k}(x)\>.
  \end{align*}
As a result, Lemma \ref{sect-2-lem-1}(iii) implies
  \begin{align*}
  \sum_{\ell=1}^\infty \frac{db_\ell}{D_\ell}\sum_{k=1}^{D_\ell}
  \big\<x-y,A_{\ell,k}(x)-A_{\ell,k}(y)\big\>^2
  &= 2\sin^2\theta \sum_{\ell=1}^\infty b_\ell[1-\gamma_\ell(\cos\theta)]
  \end{align*}
from which we obtain the second identity.
\end{proof}

In the sequel, we shall denote the two functions on the right hand sides by $G_1(\theta)$
and $G_2(\theta)$, respectively. Since $G\in C^2([0,\pi])$ fulfils \eqref{G-theta.1},
these equalities imply that the quantities on the left hand sides vanish as
the distance $\theta=d(x,y)$ goes to 0.

In the rest of this section, we restrict ourselves to the two dimensional unit sphere $S^2$
and compute the expressions of Jacobi fields on it. Let $x,y\in S^2\,(\subset\R^3)$,
$x\neq \pm y$, and $\theta=d(x,y)$ the angle between them. Assume $X\in T_x S^2,\, Y\in T_y S^2$
which are seen as vectors in $\R^3$. Let $\gamma:[0,1]\to S^2$ be the minimal geodesic
from $x$ to $y$, and $J$ the Jacobi field along
$\gamma$ satisfying $J(0)=X$ and $J(1)=Y$. To compute the expression of $J$, we write $e(a)=\frac{\dot\gamma(a)}{|\dot\gamma(a)|}=\frac{\dot\gamma(a)}{\theta}$
for the tangent vector field of $\gamma$. In particular,
  \begin{equation}\label{sect-2-Jacobi.1}
  e(0)=\frac{y-\<y,x\> x}{|y-\<y,x\> x|}=\frac{y-(\cos\theta) x}{\sin\theta},\quad
  e(1)=\frac{(\cos\theta) y-x}{\sin\theta}.
  \end{equation}
It is clear that
  \begin{equation}\label{sect-2-Jacobi.2}
  \<e(0),e(1)\>=\<x,y\>=\cos\theta.
  \end{equation}
Set $N=\frac{x\times y}{\sin\theta}$ (here $\times$ is the vector product in $\R^3$);
then $T_{\gamma(a)} S^2=\textup{span}\{e(a),N\},\, a\in[0,1]$. The fact that the normal vector
field $N$ along $\gamma$ is independent on $a\in[0,1]$ is important for us, since we do not
need to parallel-transport tangent vectors which are normal to $\gamma$.
Now we can write $J(a)=J_1(a)e(a) + J_2(a)N$ and it remains to determine the coefficients
$J_1(a)$ and $J_2(a)$. It is well known that $J_1(a)$ is a linear function of $a\in [0,1]$.
Using the boundary values of $J$ and \eqref{sect-2-Jacobi.1}, it is easy to obtain
  \begin{equation}\label{sect-2-Jacobi.3}
  J_1(a)=\frac{1-a}{\sin\theta}\<X,y\>-\frac{a}{\sin\theta}\<Y,x\>.
  \end{equation}
Next, there exist two constants $\lambda,\mu\in\R$ such that $J_2(a)=\lambda\sin(a\theta)
+\mu\cos(a\theta)$. Letting $a=0$ and $a=1$, we get two equations:
  $$\mu=J_2(0)=\<X,N\>,\quad \lambda\sin\theta+\mu\cos\theta=J_2(1)=\<Y,N\>.$$
Solving them yields
  \begin{equation}\label{sect-2-Jacobi.4}
  J_2(a)=\frac{\sin(a\theta)}{\sin\theta}\<Y,N\>+\big[\cos(a\theta)-\cot\theta\sin(a\theta)\big]\<X,N\>.
  \end{equation}

\section{Stochastic Lagrangian flow on the group of homeomorphisms} \label{sect-SLF}

Given $u:[0,T]\to \G_V^0$, then for each $t\in[0,T]$, $u(t)$ is a divergence free
vector field on $S^d$. We consider the SDE
  \begin{align}\label{sect-3.1}
  \d g_u(t)&=\bigg(u(t)\,\d t+\sqrt{\frac{\nu}{c}}\circ \d W(t)\bigg)(g_u(t)),\quad g_u(0)=e,
  \end{align}
where $\nu>0$ is the viscosity of the fluid and $c=\frac 12\sum_{\ell\geq1}b_\ell$.
First we compute the generator of equation \eqref{sect-3.1}.

\begin{proposition}\label{sect-3-prop-1}
The infinitesimal generator of the process $g_u(t)$, when computed on the functions
$F(g)(x)=f(g(x)),\,f\in C^2(S^d)$, is given by
  $$\L F=\nu \Delta f+\<u,\nabla f\>.$$
\end{proposition}

\begin{proof}
The equation \eqref{sect-3.1} can be rewritten as
  \begin{align*}
  \d g_u(t,x)&=u(t,g_u(t,x))\,\d t+\sqrt{\frac{\nu}{c}}\sum_{\ell\geq1}\sqrt{\frac{db_\ell}{D_\ell}}
  \sum_{k=1}^{D_\ell} A_{\ell,k}(g_u(t,x))\circ \d w_{\ell,k}(t),\quad g_u(0,x)=x\in S^d.
  \end{align*}
For $f\in C^2(S^d)$, It\^o's formula yields (we write $g(t)$ instead of $g_u(t,x)$ to simplify notations)
  \begin{align*}
  \d f(g(t))&=(u(t)f)(g(t))\,\d t+\sqrt{\frac{\nu}{c}}\sum_{\ell\geq1}\sqrt{\frac{db_\ell}{D_\ell}}
  \sum_{k=1}^{D_\ell} (A_{\ell,k}f)(g(t))\circ \d w_{\ell,k}(t),
  \end{align*}
where $u(t)f$ and $A_{\ell,k}f$ are the Lie derivatives of $f$. Transforming the Stratonovich
differential into the It\^o differential, we get
  \begin{align*}
  \d f(g(t))&=(u(t)f)(g(t))\,\d t+\sqrt{\frac{\nu}{c}}\sum_{\ell\geq1}\sqrt{\frac{db_\ell}{D_\ell}}
  \sum_{k=1}^{D_\ell} (A_{\ell,k}f)(g(t))\, \d w_{\ell,k}(t)\cr
  &\hskip12pt +\frac\nu{2c}\sum_{\ell\geq1}\frac{db_\ell}{D_\ell}\sum_{k=1}^{D_\ell} A_{\ell,k}(A_{\ell,k}f)(g(t))\,\d t.
  \end{align*}
Since $A_{\ell,k}(A_{\ell,k}f)=\<\nabla_{A_{\ell,k}}A_{\ell,k},\nabla f\>+\<A_{\ell,k},\nabla_{A_{\ell,k}}\nabla f\>$,
by Lemma \ref{sect-2-lem-1} (see also the proof of \cite[Theorem 6.1]{FangLuo07}), the last term reduces to
  $$\frac\nu{2c}\sum_{\ell\geq1}b_\ell\, \Delta f(g(t))\,\d t=\nu \Delta f(g(t))\,\d t.$$
The proof is complete.
\end{proof}

Next, since $W(t)$ is an isotropic Gaussian vector field with covariance function given by
\eqref{covariance}, we can apply \cite[Theorem 3.2]{AC12b} to get the variational
principle for the Navier--Stokes equation on $S^d$. We omit it to save space.

As mentioned in the beginning of \cite[Section 3]{CiprianoCruzeiro07}, the well-posedness
of equation \eqref{sect-3.1} relies on the regularity of the Brownian motion $W(t)$ and
the drift $u(t)$, see \cite{Fang04} for related studies on the unit circle $S^1$.
Here we are concerned with the case where $u\in L^2([0,T];\G_V^1)$, that is, $u$ is
a time-dependent vector field which is divergence free and has $H^1$-spatial regularity.
This is motivated by the DiPerna--Lions theory which deals with the existence and uniqueness of
quasi-invariant flows associated to weakly differentiable vector fields. It has attracted
intensive attention in the past three decades (cf. \cite{DiPernaLions89, Ambrosio04,
CrippadeLellis}), and was extended in \cite{FangLiLuo} to non-compact manifolds
under suitable curvature conditions, see \cite{Zhang} for the study of Stratonovich SDE
on compact manifolds with Sobolev drift coefficient. In our setting, we have

\begin{theorem}[Stochastic Lagrangian flow]\label{DPL-flow}
Assume that the function $G(\theta)$ defined in \eqref{G-theta} is twice continuously
differentiable and $u\in L^2([0,T];\G_V^1)$. Then the SDE \eqref{sect-3.1} generates a
unique flow $g_u(t)$ of measurable homeomorphisms on $S^d$ which preserve the volume measure.
\end{theorem}

\begin{proof}
The assertion can be proved by following the arguments of \cite[Theorem 2.5]{Zhang}. Since
$u\in L^2([0,T];\G_V^1)$, we have the decomposition below
  $$u(t,x)=\sum_{\ell\geq 1} (1+c_{\ell,\delta})^{-1/2}\sum_{k=1}^{D_\ell} u_{\ell,k}(t) A_{\ell, k}(x),$$
where $\{u_{\ell,k}:k=1,\ldots,D_\ell,\ell\geq 1\}\subset L^2([0,T])$ satisfies $\sum_{\ell\geq 1}
\sum_{k=1}^{D_\ell} \|u_{\ell,k}\|_{L^2([0,T])}<+\infty$. Note that we do not need the boundedness
of the vector field $u$ (unlike \cite[Theorem 2.5]{Zhang}), since, instead of approximating $u$
by convoluting it with a sequence of standard kernel (see \cite[Proposition 3.9]{Zhang}),
we can use the smooth approximations:
  $$u_n(t,x)=\sum_{\ell=1}^n (1+c_{\ell,\delta})^{-1/2}\sum_{k=1}^{D_\ell} u_{\ell,k}(t) A_{\ell, k}(x),\quad n\geq 1.$$
Similarly, we define
  $$W_n(t,x)=\sum_{\ell=1}^n \sqrt{\frac{db_\ell}{D_\ell}}
  \sum_{k=1}^{D_\ell}w_{\ell,k}(t)A_{\ell,k}(x)$$
and consider the SDE
  \begin{equation}\label{sect-3.2}
  \d g_n(t)=\bigg(u_n(t)\,\d t+\sqrt{\frac\nu c}\circ\d W_n(t)\bigg)(g_n(t)),\quad g_n(0)=e.
  \end{equation}
As the vector fields $u_n(t)$ and $W_n(t)$ are (a.s.) divergence free and smooth in the spatial
variable, by the classical results of Kunita \cite{Kunita90}, the above equation determines
a stochastic flow $g_n(t)$ of diffeomorphisms on $S^d$, leaving the volume measure invariant.
Following the proof of \cite[Theorem 2.5]{Zhang} we conclude that there exists a flow $g_u(t)$
of maps such that
  \begin{equation}\label{sect-3.3}
  \lim_{n\to\infty}\E\int_{S^d}\sup_{0\leq t\leq T} d^2\big(g_n(t,x),g_u(t,x)\big)\,\d x=0.
  \end{equation}
It is easy to show that the flow $g_u(t)$ preserves the volume measure of $S^d$ and solves SDE
\eqref{sect-3.1}.

To show that $g_u(t)$ admits a measurable inverse map, we fix any $t_0\in[0,T]$ and define
$w_{\ell,k}^{t_0}(s)=w_{\ell,k}(t_0-s)-w_{\ell,k}(t_0)$ for all $k=1,\ldots,D_\ell$ and $\ell\geq 1$.
Consider
  \begin{equation}\label{sect-3.4}
  \d g_n^{t_0}(s)=\bigg(-u_n(t_0-s)\,\d s+\sqrt{\frac\nu c}\circ\d W_n^{t_0}(s)\bigg) \big(g_n^{t_0}(s)\big),
  \quad g_n^{t_0}(0)=e,\ s\in[0,t_0],
  \end{equation}
where $W_n^{t_0}(s)=\sum_{\ell=1}^n \sqrt{\frac{db_\ell}{D_\ell}}\sum_{k=1}^{D_\ell}w_{\ell,k}^{t_0}(s)A_{\ell,k}$.
Then it is well known that, a.s., for all $s\in [0,t_0]$ and $x\in S^d$,
  $$g_n\big(s,g_n^{t_0}(t_0,x)\big)=g_n^{t_0}(t_0-s,x)\quad \mbox{and}\quad
  g_n^{t_0}\big(s,g_n(t_0,x)\big)=g_n(t_0-s,x).$$
In particular, letting $s=t_0$ yields
  \begin{equation}\label{sect-3.5}
  g_n\big(t_0,g_n^{t_0}(t_0,x)\big)=x=g_n^{t_0}\big(t_0,g_n(t_0,x)\big)\quad \mbox{for all } x\in S^d.
  \end{equation}
That is to say, $g_n^{-1}(t_0)=g_n^{t_0}(t_0)$. As in \eqref{sect-3.3}, there exists
a flow of maps $g^{t_0}(s)$ such that
  \begin{equation}\label{sect-3.6}
  \lim_{n\to\infty}\E\int_{S^d}\sup_{0\leq s\leq t_0} d^2\big(g_n^{t_0}(s,x),g^{t_0}(s,x)\big)\,\d x=0.
  \end{equation}
Now for any $f,h\in C(S^d)$, by \eqref{sect-3.5}, we have a.s.
  \begin{align*}
  \int_{S^d} f(g_n(t_0,x))h(x)\,\d x&=\int_{S^d} f(g_n(t_0,x)) h\big(g_n^{t_0}(t_0,g_n(t_0,x))\big)\,\d x\\
  &=\int_{S^d} f(y) h\big(g_n^{t_0}(t_0,y)\big)\,\d y,
  \end{align*}
where in the second equality we have used the invariance of the volume measure under the transform $g_n(t_0)$.
By \eqref{sect-3.3} and \eqref{sect-3.6}, letting $n\to\infty$ in the above identity implies that
for any $f,h\in C(S^d)$, it holds a.s.
  $$\int_{S^d} f(g(t_0,x))h(x)\,\d x=\int_{S^d} f(y) h\big(g^{t_0}(t_0,y)\big)\,\d y.$$
As the space $C(S^d)$ is separable, we conclude that, almost surely, the above equality holds for all
$f,h\in C(S^d)$. Then by \cite[Lemma 4.3]{Zhang} , we finish the proof.
\end{proof}

\section{The distance between two stochastic Lagrangian flows on the group $G_V^0$} \label{sect-distance}

In this section we consider the distance between two stochastic Lagrangian flows on the
group of volume-preserving homeomorphisms of the sphere $S^d$. We shall derive an
equation for the distance between two flows with smooth drift $u(t,x)$. Remark that
Arnaudon and Cruzeiro proved in \cite[Proposition 6.1]{AC12b} such an formula in the
general setting of compact Riemannian manifolds without boundary. Using
the particular properties of our vector fields $A_{\ell,k}$ (see Lemma \ref{sect-2-lem-1}),
we shall do some explicit computations. To avoid the difficulty of cut-locus,
we use the extrinsic distance on $S^d$.

Consider two flows on the group of homeomorphisms:
  \begin{equation}\label{sect-4.1}
  \d g_t(x)=u(t,g_t(x))\,\d t+\sqrt{\frac\nu c}\circ \d W(t,g_t(x)),\quad g_0=\varphi\in G_V^0
  \end{equation}
and
  \begin{equation}\label{sect-4.2}
  \d \tilde g_t(x)=u(t,\tilde g_t(x))\,\d t+\sqrt{\frac\nu c}\circ \d W(t,\tilde g_t(x)),
  \quad \tilde g_0=\psi\in G_V^0.
  \end{equation}
We fix an orthonormal basis $\{\theta_1,\ldots,\theta_{d+1}\}$ of $\R^{d+1}$, and for $1\leq i\leq d+1$,
let $\xi^i_t=\<\theta_i,g_t(x)\>$ and $\zeta^i_t=\<\theta_i,\tilde g_t(x)\>$.

\begin{lemma}\label{4-lem}
We have
  \begin{equation}\label{4-lem.1}
  \d\xi^i_t=\<\theta_i,u(t,g_t(x))\>\,\d t+\sqrt{\frac\nu c}\sum_{\ell=1}^\infty\sqrt{\frac{db_\ell}{D_\ell}}
  \sum_{k=1}^{D_\ell}\<\theta_i,A_{\ell,k}(g_t(x))\>\,\d w_{\ell,k}(t)-d\nu\xi^i_t\,\d t.
  \end{equation}
\end{lemma}

\begin{proof}
It\^o's formula yields
  \begin{equation}\label{4-lem.2}
  \d\xi^i_t=\big\<\theta_i,u(t,g_t(x))\big\>\,\d t
  +\sqrt{\frac\nu c}\big\<\theta_i,\circ\, \d W(t,g_t(x))\big\>.
  \end{equation}
We need to transform the second term
  $$\sqrt{\frac\nu c}\big\<\theta_i,\circ\, \d W(t,g_t(x))\big\>
  =\sqrt{\frac\nu c}\sum_{\ell\geq 1} \sqrt{\frac{db_\ell}{D_\ell}}\sum_{k=1}^{D_\ell}
  \big\<\theta_i,A_{\ell,k}(g_t(x))\big\> \circ \d w_{\ell,k}(t)$$
into the It\^o stochastic differential. Let $Q_x:\R^{d+1}\to T_x S^d$ be the orthogonal
projection and $\Lambda_t=Q_{g_t(x)}\theta_i$. We have
  $$\d\big\<\theta_i,A_{\ell,k}(g_t(x))\big\>=\d\big\<\Lambda_t,A_{\ell,k}(g_t(x))\big\>
  =\Big\<\frac{\rm D}{\d t}\Lambda_t,A_{\ell,k}(g_t(x))\Big\>
  +\Big\<\Lambda_t,\frac{\rm D}{\d t}A_{\ell,k}(g_t(x))\Big\>,$$
where $\frac{\rm D}{\d t}$ is the covariant derivative along $\{g_t(x)\}_{t\geq0}$.
First, the It\^o contraction
  $$\Big\<\Lambda_t,\frac{\rm D}{\d t}A_{\ell,k}(g_t(x))\Big\> \cdot \d w_{\ell,k}(t)
  =\sqrt{\frac{\nu db_\ell}{c D_\ell}} \big\<\Lambda_t,(\nabla_{A_{\ell,k}} A_{\ell,k})(g_t(x))\big\>\,\d t.$$
Next, noticing that $\Lambda_t=\theta_i-\<\theta_i,g_t(x)\> g_t(x)$, we have
  $$\d\Lambda_t\cdot \d w_{\ell,k}(t)=-\sqrt{\frac{\nu db_\ell}{c D_\ell}}
  \, \big[\big\<\theta_i,A_{\ell,k}(g_t(x))\big\> g_t(x)
  + \<\theta_i,g_t(x)\>A_{\ell,k}(g_t(x))\big]\,\d t.$$
Consequently,
  $$\Big\<\frac{\rm D}{\d t}\Lambda_t,A_{\ell,k}(g_t(x))\Big\> \cdot \d w_{\ell,k}(t)
  = -\sqrt{\frac{\nu db_\ell}{c D_\ell}}\<\theta_i,g_t(x)\> |A_{\ell,k}(g_t(x))|_{\R^{d+1}}^2\,\d t.$$
Recalling that $\xi^i_t=\<\theta_i,g_t(x)\>$, we obtain
  $$\d\big\<\theta_i,A_{\ell,k}(g_t(x))\big\> \cdot \d w_{\ell,k}(t)= \sqrt{\frac{\nu db_\ell}{c D_\ell}}
  \, \big[\big\<\Lambda_t,(\nabla_{A_{\ell,k}} A_{\ell,k})(g_t(x))\big\>
  -\xi^i_t |A_{\ell,k}(g_t(x))|_{\R^{d+1}}^2 \big]\d t.$$
Therefore by Lemma \ref{sect-2-lem-1}(i) and (ii),
  \begin{align*}
  \sqrt{\frac\nu c}\big\<\theta_i,\circ\, \d W(t,g_t(x))\big\>
  &=\sqrt{\frac\nu c}\sum_{\ell\geq 1} \sqrt{\frac{db_\ell}{D_\ell}}\sum_{k=1}^{D_\ell}
  \big\<\theta_i,A_{\ell,k}(g_t(x))\big\> \,\d w_{\ell,k}(t)
  -d\nu\xi^i_t \d t,
  \end{align*}
since $c=\frac12\sum_{\ell\geq 1} b_\ell$. Substituting this equality into \eqref{4-lem.2}
completes the proof.
\end{proof}

Replacing $g_t(x)$ with $\tilde g_t(x)$ in \eqref{4-lem.1}, we get the equation for $\zeta^i_t$. Therefore
  \begin{align*}
  \d(\xi^i_t-\zeta^i_t)&=\big\<\theta_i,u(t,g_t(x))-u(t,\tilde g_t(x))\big\>\,\d t
  -d\nu(\xi^i_t-\zeta^i_t)\,\d t\cr
  &\hskip12pt +\sqrt{\frac\nu c}\sum_{\ell=1}^\infty\sqrt{\frac{db_\ell}{D_\ell}}
  \sum_{k=1}^{D_\ell}\big\<\theta_i,A_{\ell,k}(g_t(x))-A_{\ell,k}(\tilde g_t(x))\big\>\,\d w_{\ell,k}(t).
  \end{align*}
The It\^o formula leads to
  \begin{align*}
  \d\big[(\xi^i_t-\zeta^i_t)^2\big]&=2(\xi^i_t-\zeta^i_t)\big\<\theta_i,u(t,g_t(x))-u(t,\tilde g_t(x))\big\>\,\d t
  -2d\nu(\xi^i_t-\zeta^i_t)^2\,\d t\cr
  &\hskip12pt +\frac\nu c\sum_{\ell=1}^\infty \frac{db_\ell}{D_\ell} \sum_{k=1}^{D_\ell}
  \big\<\theta_i,A_{\ell,k}(g_t(x))-A_{\ell,k}(\tilde g_t(x))\big\>^2\,\d t\cr
  &\hskip12pt +2\sqrt{\frac\nu c}\sum_{\ell=1}^\infty\sqrt{\frac{db_\ell}{D_\ell}}\sum_{k=1}^{D_\ell}
  (\xi^i_t-\zeta^i_t)\big\<\theta_i,A_{\ell,k}(g_t(x))-A_{\ell,k}(\tilde g_t(x))\big\>\,\d w_{\ell,k}(t).
  \end{align*}
Since $\{\theta_1,\ldots,\theta_{d+1}\}$ is an orthonormal basis of $\R^{d+1}$, summing from $i=1$
to $d+1$ gives us
  \begin{align*}
  \d|g_t(x)-\tilde g_t(x)|^2&=2\big\<g_t(x)-\tilde g_t(x),u(t,g_t(x))-u(t,\tilde g_t(x))\big\>\,\d t
  -2d\nu|g_t(x)-\tilde g_t(x)|^2\,\d t\cr
  &\hskip12pt +\frac\nu c\sum_{\ell=1}^\infty \frac{db_\ell}{D_\ell} \sum_{k=1}^{D_\ell}
  \big|A_{\ell,k}(g_t(x))-A_{\ell,k}(\tilde g_t(x))\big|^2\,\d t\cr
  &\hskip12pt +2\sqrt{\frac\nu c}\sum_{\ell=1}^\infty\sqrt{\frac{db_\ell}{D_\ell}}\sum_{k=1}^{D_\ell}
  \big\<g_t(x)-\tilde g_t(x),A_{\ell,k}(g_t(x))-A_{\ell,k}(\tilde g_t(x))\big\>\,\d w_{\ell,k}(t).
  \end{align*}
Denote by $\rho_t(x)=d(g_t(x),\tilde g_t(x))$ the intrinsic distance (or angle) between
$g_t(x),\tilde g_t(x)$, and $\beta_t(x)=|g_t(x)-\tilde g_t(x)|$ the extrinsic distance.
Using Corollary \ref{sect-2-cor} and recalling that $G_1(\theta)=2d
\big[G(0)-\cos\theta\, G(\theta)\big]- 2\sin\theta\, G'(\theta)$, we have
  \begin{align*}
  \d\beta_t^2(x)&=\bigg[2\big\<g_t(x)-\tilde g_t(x),u(t,g_t(x))-u(t,\tilde g_t(x))\big\>
  -2d\nu\beta_t^2(x)+\frac\nu c G_1(\rho_t(x))\bigg] \d t\cr
  &\hskip12pt +2\sqrt{\frac\nu c}\sum_{\ell=1}^\infty\sqrt{\frac{db_\ell}{D_\ell}}\sum_{k=1}^{D_\ell}
  \big\<g_t(x)-\tilde g_t(x),A_{\ell,k}(g_t(x))-A_{\ell,k}(\tilde g_t(x))\big\>\,\d w_{\ell,k}(t).
  \end{align*}

We define $\gamma_t=\big[\int_{S^d}\beta_t^2(x)\,\d x\big]^{1/2}$. Then integrating both sides of the above equality on $S^d$ leads to
  \begin{align*}
  \d\gamma_t^2&=\bigg[2\big\<g_t-\tilde g_t,u(t,g_t)-u(t,\tilde g_t)\big\>_{S^d}
  +\frac\nu c \int_{S^d}G_1(\rho_t(x))\,\d x-2d\nu \gamma_t^2\bigg]\d t\cr
  &\hskip12pt +2\sqrt{\frac\nu c}\sum_{\ell=1}^\infty\sqrt{\frac{db_\ell}{D_\ell}}\sum_{k=1}^{D_\ell}
  \big\<g_t-\tilde g_t,A_{\ell,k}(g_t)-A_{\ell,k}(\tilde g_t)\big\>_{S^d}\,\d w_{\ell,k}(t),
  \end{align*}
where $\<\cdot,\cdot\>_{S^d}$ denotes the inner product in $L^2(S^d,\d x)$.
We write $M(t)$ for the martingale part. Then It\^o's formula yields
  \begin{align*}
  \d\gamma_t&=\frac1{\gamma_t}\bigg[\big\<g_t-\tilde g_t,u(t,g_t)-u(t,\tilde g_t)\big\>_{S^d}
  +\frac\nu{2c} \int_{S^d}G_1(\rho_t(x))\,\d x-d\nu \gamma_t^2\bigg]\d t\cr
  &\hskip12pt +\frac1{2\gamma_t}\,\d M(t)-\frac{\nu}{2c\gamma_t^3}\sum_{\ell=1}^\infty \frac{db_\ell}{D_\ell}
  \sum_{k=1}^{D_\ell}\big\<g_t-\tilde g_t,A_{\ell,k}(g_t)-A_{\ell,k}(\tilde g_t)\big\>_{S^d}^2\,\d t.
  \end{align*}
As in \cite[Section 4]{AC12b}, we introduce the notations
  $$n_g(t)=\frac{g_t-\tilde g_t}{\gamma_t},\quad \delta u(t)=\frac{u(t,g_t)-u(t,\tilde g_t)}{\gamma_t}
  \quad\mbox{and}\quad \delta A_{\ell,k}(t)=\frac{A_{\ell,k}(g_t)-A_{\ell,k}(\tilde g_t)}{\gamma_t}.$$
Then the above equation can be reduced to
  \begin{align}\label{sect-4.3}
  \d\gamma_t&=\gamma_t\bigg[\frac1{2\gamma_t^2}\,\d M(t)+\Big(\big\<n_g(t),\delta u(t)\big\>_{S^d}-d\nu
  +\frac\nu{2c\gamma_t^2} \int_{S^d}G_1(\rho_t(x))\,\d x\cr
  &\hskip28pt -\frac{\nu}{2c}\sum_{\ell=1}^\infty \frac{db_\ell}{D_\ell}
  \sum_{k=1}^{D_\ell}\big\<n_g(t),\delta A_{\ell,k}(t)\big\>_{S^d}^2\Big)\d t\bigg].
  \end{align}
Using our notations, the martingale part can be rewritten as
  $$\d\tilde M(t):=\frac1{2\gamma_t^2}\,\d M(t)=\sqrt{\frac\nu c}\sum_{\ell=1}^\infty
  \sqrt{\frac{db_\ell}{D_\ell}}\sum_{k=1}^{D_\ell}\<n_g(t),\delta A_{\ell,k}(t)\>_{S^d}\,\d w_{\ell,k}(t).$$
Its quadratic variation is given by
  $$\d\<\tilde M\>(t)=\frac{\nu}c \sum_{\ell=1}^\infty\frac{db_\ell}{D_\ell}
  \sum_{k=1}^{D_\ell}\<n_g(t),\delta A_{\ell,k}(t)\>_{S^d}^2\,\d t.$$
By Cauchy's inequality and Corollary \ref{sect-2-cor}, we have
  \begin{align}\label{sect-4.4}
  \d\<\tilde M\>(t)&\leq \frac{\nu}{c\gamma_t^2} \sum_{\ell=1}^\infty\frac{db_\ell}{D_\ell}
  \sum_{k=1}^{D_\ell}\bigg[\int_{S^d}\big|A_{\ell,k}(g_t(x))-A_{\ell,k}(\tilde g_t(x))\big|^2\,\d x\bigg]\d t\cr
  &=\frac{\nu}{c\gamma_t^2} \bigg[\int_{S^d} G_1(\rho_t(x))\,\d x\bigg] \d t.
  \end{align}
Since $G_1(\theta)\leq C_0\theta^2$ for some $C_0>0$, we deduce from \eqref{sect-2.0} that
  \begin{equation*}
  \d\<\tilde M\>(t)\leq \frac{C_0 \nu\, \pi^2}{4c}\,\d t.
  \end{equation*}
Therefore we have proved the following result which is analogous to
\cite[Proposition 4.2]{AC12b}.

\begin{proposition}\label{sect-4-prop-1}
Let $g_t$ and $\tilde g_t$ be two flows defined by \eqref{sect-4.1} and \eqref{sect-4.2}, respectively.
Then the distance $\gamma_t=\big[\int_{S^d}|g_t(x)-\tilde g_t(x)|^2\,\d x\big]^{1/2}$ between them
satisfies
  \begin{align*}
  \d\gamma_t&=\gamma_t\big[\sigma_t\,\d z_t+b_t\,\d t+\big\<n_g(t),\delta u(t)\big\>_{S^d}\,\d t\big],
  \end{align*}
where $z_t$ is a real valued Brownian motion, $\sigma_t>0$ is given by
  \begin{align*}
  \sigma_t^2=\frac{\d\<\tilde M\>(t)}{\d t}=\frac{\nu}c \sum_{\ell=1}^\infty\frac{db_\ell}{D_\ell}
  \sum_{k=1}^{D_\ell}\<n_g(t),\delta A_{\ell,k}(t)\>_{S^d}^2
  \end{align*}
and
  \begin{equation*}
  b_t=-d\nu +\frac\nu{2c\gamma_t^2} \int_{S^d}G_1(\rho_t(x))\,\d x
  -\frac{1}{2}\sigma_t^2,
  \end{equation*}
where $\rho_t(x)=d(g_t(x),\tilde g_t(x))$ is the intrinsic distance.
\end{proposition}

Unfortunately, it seems that there is no good method to estimate $\sigma_t^2$, except
using Cauchy's inequality. However, by \eqref{sect-4.4}, this leads to
  $$b_t\geq -d\nu.$$
As the right hand side is negative, we are unable to derive useful estimates
on $\rho_t$ from the above inequality (cf. \cite[Example 5.1]{AC12b}).

\section{The rotation process on $S^2$} \label{sect-rotation}

In this section, we consider the rotation of two particles $g_t(x)$ and $\tilde g_t(x)$
on the two dimensional sphere $S^2$ when their distance is small. In \cite[Lemma 7.1]{AC12b},
the authors established a formula for the covariant derivative of the rotation process.
To introduce this formula, we recall the notations in \cite[Sections 6 and 7]{AC12b}.
For simplicity, we denote by $x_t=g_t(x)$ and $y_t=\tilde g_t(x)$.
When $y_t$ is not in the cut-locus of $x_t$, let $[0,1]\ni a\mapsto
\gamma_a(x_t,y_t)$ be the minimal geodesic from $x_t$ to $y_t$. We denote by $T_a=T_a(t)
=\dot\gamma_a(x_t,y_t)$ and $\gamma_a(t)=\gamma_a(x_t,y_t)$. Let $e(t)\in T_{x_t} S^2$
be given by $e(t)=\frac{T_0(t)}{\rho_t(x)}$, where $\rho_t(x)=d(x_t,y_t)$ is the intrinsic
distance (or equivalently, the angle) between $x_t$ and $y_t$. Our purpose is to compute
the covariant derivative $\D e(t)$.

Fix $\ell\geq1$ and $k\in\{1,\ldots, D_\ell\}$. Let $J_{\ell,k}(a)$ be the Jacobi field
along $\gamma_a(t)$ satisfying $J_{\ell,k}(0)=A_{\ell,k}(x_t)$ and $J_{\ell,k}(1)=A_{\ell,k}(y_t)$.
Sometimes, we write $J_a(X,Y)$ for the Jacobi field along $\gamma_a(t)$ with $J_0(X,Y)
=X\in T_{x_t}S^2$ and $J_1(X,Y)=Y\in T_{y_t}S^2$.
Denote by $e_a(t)=\frac{T_a(t)}{\rho_t(x)}$ the unit tangent vector field and $N(t)=
\frac{x_t\times y_t}{\sin\rho_t(x)}$ the unit tangent vector normal to $\gamma_a(t)$.
We also write $u^N(t,x_t)$ for the part of $u(t,x_t)$ which is normal to the geodesic
$\gamma_a(t)$; moreover, $\d_m x_t^N$ is the martingale part normal to $\gamma_a(t)$.
Adapting \cite[Lemma 7.1]{AC12b} to our framework yields

\begin{lemma}\label{sect-5-lem-1}
It holds that
  \begin{align}\label{sect-5-lem-1.1}
  \hskip-20pt\D e(t)&=\frac1{\rho_t(x)}\dot J_0\big(\d_m x_t^N,\d_m y_t^N\big)
  +\frac1{\rho_t(x)}\dot J_0\big(u^N(t,x_t),u^N(t,y_t)\big)\,\d t\cr
  &\hskip13pt +\frac{\nu}{2c\rho_t(x)}\sum_{\ell\geq 1}\frac{2b_\ell}{D_\ell}\sum_{k=1}^{D_\ell}
  \big[\nabla_{T_0}\nabla_{J_{\ell,k}(0)}J_{\ell,k}(0)-R(T_0,J_{\ell,k}(0))J_{\ell,k}(0)\big]\,\d t\\
  &\hskip13pt -\frac{\nu}{2c\rho_t^2(x)}\bigg[\sum_{\ell\geq 1}\frac{2b_\ell}{D_\ell}\sum_{k=1}^{D_\ell}
  \int_0^1\!\! \big(\big|\nabla_{T_a} J_{\ell,k}(a)\big|^2
  -\big\<R(T_a,J_{\ell,k}(a))J_{\ell,k}(a),T_a\big\>\big)\,\d a\bigg]e(t)\,\d t,\nonumber
  \end{align}
where $\nu$ is the viscosity of the fluid and $c=\frac12 G(0)=\frac12\sum_{\ell\geq 1}b_\ell<\infty$.
\end{lemma}

We want to get explicit expressions of the terms in the above formula.
Compared to \cite[Proposition 7.2]{AC12b} which deals with the two dimensional flat torus,
the computations given below are more complicated due to the presence of the curvature.
First, by the discussions at the end of Section 2, $J_a\big(\d_m x_t^N,\d_m y_t^N\big)$
is a normal Jacobi field with coefficient
  $$\frac{\sin(a\rho_t(x))}{\sin\rho_t(x)}\<\d_m y_t,N(t)\>
  +\big[\cos(a\rho_t(x))-\cot\rho_t(x)\sin(a\rho_t(x))\big]\<\d_m x_t,N(t)\>.$$
Hence
  \begin{equation}\label{sect-5-term-1}
  \dot J_0\big(\d_m x_t^N,\d_m y_t^N\big)=\frac{\rho_t(x)}{\sin\rho_t(x)}
  \big(\<\d_m y_t,N(t)\>-\cos\rho_t(x)\<\d_m x_t,N(t)\>\big)N(t).
  \end{equation}
By the definitions of $\d_m x_t$ and $\d_m y_t$, we obtain
  \begin{equation}\label{sect-5-term-1.1}
  \begin{split}
  &\hskip13pt \frac1{\rho_t(x)}\dot J_0\big(\d_m x_t^N,\d_m y_t^N\big)\cr
  &=\frac{N(t)}{\sin\rho_t(x)}\sqrt{\frac\nu c}
  \sum_{\ell\geq 1}\sqrt{\frac{2b_\ell}{D_\ell}}\sum_{k=1}^{D_\ell}
  \big(\<A_{\ell,k}(y_t),N(t)\>-\<A_{\ell,k}(x_t),N(t)\>\cos\rho_t(x)\big)\,\d w_{\ell,k}(t).
  \end{split}
  \end{equation}
To compute the quadratic
variation of the above process, we have to deal with the
quantity involving $\<A_{\ell,k}(x_t),N(t)\>$. We collect the corresponding
results in the following lemma.

\begin{lemma}\label{sect-5-lem-2}
We have for any $\ell\geq 1$,
  \begin{align*}
  \frac2{D_\ell}\sum_{k=1}^{D_\ell}\<A_{\ell,k}(x_t),N(t)\>^2
  &=\frac2{D_\ell}\sum_{k=1}^{D_\ell}\<A_{\ell,k}(y_t),N(t)\>^2=1;\cr
  \frac2{D_\ell}\sum_{k=1}^{D_\ell}\<A_{\ell,k}(x_t),N(t)\>\<A_{\ell,k}(y_t),N(t)\>
  &=\cos\rho_t(x)\, \gamma_\ell(\cos\rho_t(x))-\sin^2\rho_t(x)\, \gamma'_\ell(\cos\rho_t(x)).
  \end{align*}
\end{lemma}

\begin{proof}
We first list some useful identities. By \eqref{sect-2-Jacobi.1}, we have
  $$e_0(t)=e(t)=\frac{y_t-(\cos\rho_t(x)) x_t}{\sin\rho_t(x)}\quad\mbox{and}\quad
  e_1(t)=\frac{(\cos\rho_t(x)) y_t-x_t}{\sin\rho_t(x)}.$$
As a result, $\<e_0(t),e_1(t)\>=\cos\rho_t(x)$ and
  \begin{align*}
  &\<A_{\ell,k}(x_t),e_0(t)\>=\frac{\<A_{\ell,k}(x_t),y_t\>}{\sin\rho_t(x)},\quad
  \<A_{\ell,k}(x_t),e_1(t)\>=\cot\rho_t(x) \<A_{\ell,k}(x_t),y_t\>,\cr
  &\<A_{\ell,k}(y_t),e_0(t)\>=-\cot\rho_t(x) \<A_{\ell,k}(y_t),x_t\>,\quad
  \<A_{\ell,k}(y_t),e_1(t)\>=-\frac{\<A_{\ell,k}(y_t),x_t\>}{\sin\rho_t(x)}.
  \end{align*}

Now, since $\{e_0(t),N(t)\}$ is an orthonormal basis of $T_{x_t}S^2$, one has
  \begin{align*}
  \<A_{\ell,k}(x_t),N(t)\>^2&=|A_{\ell,k}(x_t)|^2-\<A_{\ell,k}(x_t),e_0(t)\>^2
  =|A_{\ell,k}(x_t)|^2-\frac{\<A_{\ell,k}(x_t),y_t\>^2}{\sin^2\rho_t(x)},
  \end{align*}
hence by Lemma \ref{sect-2-lem-1}(ii) and (iii),
  \begin{align*}
  \frac2{D_\ell}\sum_{k=1}^{D_\ell}\<A_{\ell,k}(x_t),N(t)\>^2
  &=\frac2{D_\ell}\sum_{k=1}^{D_\ell}|A_{\ell,k}(x_t)|^2
  -\frac{1}{\sin^2\rho_t(x)}\cdot \frac2{D_\ell}\sum_{k=1}^{D_\ell}\<A_{\ell,k}(x_t),y_t\>^2\cr
  &=2-\frac{1}{\sin^2\rho_t(x)}\sin^2\rho_t(x)=1.
  \end{align*}
In the same way, we can prove the equality involving $\<A_{\ell,k}(y_t),N(t)\>^2$.

Next we denote by $I_{\ell,k}=\<A_{\ell,k}(x_t),N(t)\>\<A_{\ell,k}(y_t),N(t)\>$. Then
  \begin{align*}
  I_{\ell,k}&=\big\<\<A_{\ell,k}(x_t),N(t)\>N(t),\<A_{\ell,k}(y_t),N(t)\>N(t)\big\>\cr
  &=\big\<A_{\ell,k}(x_t)-\<A_{\ell,k}(x_t),e_0(t)\>e_0(t),A_{\ell,k}(y_t)-\<A_{\ell,k}(y_t),e_1(t)\>e_1(t)\big\>\cr
  &=\<A_{\ell,k}(x_t),A_{\ell,k}(y_t)\>-\<A_{\ell,k}(y_t),e_1(t)\>\<A_{\ell,k}(x_t),e_1(t)\>\cr
  &\hskip13pt -\<A_{\ell,k}(x_t),e_0(t)\> \<A_{\ell,k}(y_t),e_0(t)\>
  +\<A_{\ell,k}(x_t),e_0(t)\> \<A_{\ell,k}(y_t),e_1(t)\> \<e_0(t),e_1(t)\>.
  \end{align*}
Using the identities given at the beginning of the proof, we get
  \begin{align*}
  I_{\ell,k}&=\<A_{\ell,k}(x_t),A_{\ell,k}(y_t)\>+\frac{\cos\rho_t(x)}{\sin^2\rho_t(x)}
  \<A_{\ell,k}(x_t),y_t\>\<A_{\ell,k}(y_t),x_t\>.
  \end{align*}
Since
  $$2\<A_{\ell,k}(x_t),y_t\>\<A_{\ell,k}(y_t),x_t\>=\big(\<A_{\ell,k}(x_t),y_t\>+\<A_{\ell,k}(y_t),x_t\>\big)^2
  -\<A_{\ell,k}(x_t),y_t\>^2-\<A_{\ell,k}(y_t),x_t\>^2,$$
thus by Lemma \ref{sect-2-lem-1}(iii),
  \begin{align*}
  \frac2{D_\ell}\sum_{k=1}^{D_\ell}\<A_{\ell,k}(x_t),y_t\>\<A_{\ell,k}(y_t),x_t\>
  &=\frac12\big[2\sin^2\rho_t(x)\big(1-\gamma_\ell(\cos\rho_t(x))\big)-\sin^2\rho_t(x)-\sin^2\rho_t(x)\big]\cr
  &=-\sin^2\rho_t(x)\,\gamma_\ell(\cos\rho_t(x)).
  \end{align*}
Consequently, Lemma \ref{sect-2-lem-1}(ii) leads to
  $$\frac2{D_\ell}\sum_{k=1}^{D_\ell}I_{\ell,k}=\cos\rho_t(x)\gamma_\ell(\cos\rho_t(x))
  -\sin^2\rho_t(x)\,\gamma'_\ell(\cos\rho_t(x)).$$
The proof is complete.
\end{proof}

We denote by $\d\xi_t=\frac1{\rho_t(x)}\dot J_0\big(\d_m x_t^N,\d_m y_t^N\big)$ to save notations.
Then by \eqref{sect-5-term-1.1} and Lemma \ref{sect-5-lem-2},
  \begin{align*}
  \d\<\xi,\xi\>_t&=\frac{\nu}{c\sin^2\rho_t(x)}\sum_{\ell\geq 1}\frac{2b_\ell}{D_\ell}\sum_{k=1}^{D_\ell}
  \big(\<A_{\ell,k}(y_t),N(t)\>-\<A_{\ell,k}(x_t),N(t)\>\cos\rho_t(x)\big)^2\,\d t\cr
  &=\frac{\nu}{c\sin^2\rho_t(x)}\sum_{\ell\geq 1}b_\ell \big\{1+\cos^2\rho_t(x)-2\cos\rho_t(x)\cr
  &\hskip40pt \times \big[\cos\rho_t(x)\gamma_\ell(\cos\rho_t(x))
  -\sin^2\rho_t(x)\,\gamma'_\ell(\cos\rho_t(x))\big]\big\}\d t,
  \end{align*}
which, together with the definition \eqref{G-theta} of the function $G(\theta)$, gives us
  \begin{align}\label{sect-5-term-1.2}
  \d\<\xi,\xi\>_t
  &=\frac{\nu}{c}\bigg\{\frac{1+\cos^2\rho_t(x)}{\sin^2\rho_t(x)}G(0)-2\cot^2\rho_t(x)\, G(\rho_t(x))
  -2\cot\rho_t(x)\, G'(\rho_t(x))\bigg\}\d t\cr
  &= \bigg\{2\nu+\frac{2\nu}c \cot^2\rho_t(x)\,[G(0)-G(\rho_t(x))]-\frac{2\nu}c\cot\rho_t(x)\, G'(\rho_t(x))\bigg\}\,\d t.
  \end{align}

Analogous to \eqref{sect-5-term-1}, the second term on the right hand side of \eqref{sect-5-lem-1.1} is given by
  \begin{equation}\label{sect-5-term-2}
  \frac1{\rho_t(x)}\dot J_0\big(u^N(t,x_t),u^N(t,y_t)\big)
  =\frac{N(t)}{\sin\rho_t(x)}\big(\<u(t,y_t),N(t)\>-\<u(t,x_t),N(t)\> \cos\rho_t(x)\big).
  \end{equation}

Next we compute the Jacobi field $J_{\ell,k}(a)$ which are needed for treating the last
two terms in \eqref{sect-5-lem-1.1}. By the discussions at the end of Section 2, we have $J_{\ell,k}(a)=J_{\ell,k}^{(1)}(a)e_a(t)+J_{\ell,k}^{(2)}(a)N(t)$, where
  \begin{align*}
  J_{\ell,k}^{(1)}(a)&=\frac{1-a}{\sin\rho_t(x)}\<A_{\ell,k}(x_t), y_t\>
  -\frac{a}{\sin\rho_t(x)}\<A_{\ell,k}(y_t), x_t\>,\cr
  J_{\ell,k}^{(2)}(a)&=\frac{\sin(a\rho_t(x))}{\sin\rho_t(x)}\<A_{\ell,k}(y_t),N(t)\>
  +\big[\cos(a\rho_t(x))-\cot\rho_t(x)\sin(a\rho_t(x))\big]\<A_{\ell,k}(x_t),N(t)\>.
  \end{align*}
It is known that $\nabla_{T_0}\nabla_{J_{\ell,k}(0)}J_{\ell,k}(0)$ vanishes (see also \cite[p.372]{AC12a}),
so we now consider the term $R(T_0,J_{\ell,k}(0))J_{\ell,k}(0)$. As $S^2$ has constant sectional curvature
1, we have
  $$R(T_0,J_{\ell,k}(0))J_{\ell,k}(0)=-\<T_0,J_{\ell,k}(0)\>J_{\ell,k}(0)
  +\<J_{\ell,k}(0),J_{\ell,k}(0)\>T_0.$$
Next, since $T_0=\rho_t(x)e(t)$ and $J_{\ell,k}(0)=\frac{1}{\sin\rho_t(x)}\<A_{\ell,k}(x_t), y_t\>e(t)
+\<A_{\ell,k}(x_t),N(t)\> N(t)$, we obtain
  \begin{equation*}
  \begin{split}
  R(T_0,J_{\ell,k}(0))J_{\ell,k}(0)&=-\frac{\rho_t(x)}{\sin\rho_t(x)}\<A_{\ell,k}(x_t), y_t\>
  \bigg[\frac{\<A_{\ell,k}(x_t), y_t\>}{\sin\rho_t(x)}e(t)+\<A_{\ell,k}(x_t),N(t)\> N(t)\bigg]\cr
  &\hskip13pt +\rho_t(x)\bigg[\frac{\<A_{\ell,k}(x_t), y_t\>^2}{\sin^2\rho_t(x)}+\<A_{\ell,k}(x_t),N(t)\>^2\bigg]e(t)\cr
  &=\rho_t(x)\<A_{\ell,k}(x_t),N(t)\>^2 e(t)
  -\frac{\rho_t(x)}{\sin\rho_t(x)}\<A_{\ell,k}(x_t), y_t\>\<A_{\ell,k}(x_t),N(t)\>N(t).
  \end{split}
  \end{equation*}
By Lemma \ref{sect-5-lem-2}, we have
  \begin{equation*}
  \frac2{D_\ell}\sum_{k=1}^{D_\ell}R(T_0,J_{\ell,k}(0))J_{\ell,k}(0)
  =\rho_t(x)e(t)-\frac{2\rho_t(x)}{D_\ell}\sum_{k=1}^{D_\ell}
  \<A_{\ell,k}(x_t), e_0(t)\>\<A_{\ell,k}(x_t),N(t)\>N(t).
  \end{equation*}
As a result,
  \begin{align}\label{sect-5-term-3}
  &\hskip13pt \frac{\nu}{2c\rho_t(x)}\sum_{\ell\geq 1}\frac{2b_\ell}{D_\ell}
  \sum_{k=1}^{D_\ell}\big[\nabla_{T_0}\nabla_{J_{\ell,k}(0)}J_{\ell,k}(0)-R(T_0,J_{\ell,k}(0))J_{\ell,k}(0)\big]\cr
  &=-\nu\, e(t)+ \frac\nu c \sum_{\ell\geq 1}\frac{b_\ell}{D_\ell}\sum_{k=1}^{D_\ell}
  \<A_{\ell,k}(x_t), e_0(t)\>\<A_{\ell,k}(x_t),N(t)\>N(t).
  \end{align}
Unfortunately, we are unable to simplify the coefficient of the normal part $N(t)$, but we have
a good estimate on it. In fact, by Cauchy's inequality and Lemma \ref{sect-5-lem-2},
  $$\frac{2}{D_\ell}\sum_{k=1}^{D_\ell}\<A_{\ell,k}(x_t), e_0(t)\>\<A_{\ell,k}(x_t),N(t)\>\leq 1,$$
which implies
  \begin{equation}\label{sect-5-term-3.1}
  \frac\nu c \sum_{\ell\geq 1}\frac{b_\ell}{D_\ell}\sum_{k=1}^{D_\ell}
  \<A_{\ell,k}(x_t), e_0(t)\>\<A_{\ell,k}(x_t),N(t)\> \leq \frac\nu{2c} G(0)=\nu.
  \end{equation}

Finally we handle with the last term in \eqref{sect-5-lem-1.1}. As the calculation is very long we
put it in the following lemma.

\begin{lemma}\label{sect-5-lem-3}
We have
  \begin{equation}\label{sect-5-lem-3.1}
  \begin{split}
  &\hskip13pt \frac{\nu}{2c\rho_t^2(x)}\sum_{\ell\geq 1}\frac{2b_\ell}{D_\ell}\sum_{k=1}^{D_\ell}
  \int_0^1 \big(\big|\nabla_{T_a} J_{\ell,k}(a)\big|^2- \big\<R(T_a,J_{\ell,k}(a))J_{\ell,k}(a),T_a\big\>\big)\,\d a\cr
  &=\frac\nu c [G(0)-G(\rho_t(x))]\bigg\{1+\frac{\cot^2\rho_t(x)}2
  \bigg[\rho_t^2(x)\bigg(1+\frac{\sin(2\rho_t(x))}{2\rho_t(x)}\bigg)
  -\bigg(1-\frac{\sin(2\rho_t(x))}{2\rho_t(x)}\bigg)\bigg]\cr
  &\hskip40pt +\frac{1-\cos(2\rho_t(x))}{4\rho_t(x)}[1+\rho_t^2(x)]\cot\rho_t(x)\bigg\}\cr
  &\hskip13pt +\nu [\rho_t^2(x)-1] -\frac\nu c G'(\rho_t(x))\frac{1-\cos(2\rho_t(x))}{4\rho_t(x)}[1+\rho_t^2(x)]\cr
  &\hskip13pt -\frac\nu c G'(\rho_t(x))\frac{\cot\rho_t(x)}2\bigg[\rho_t^2(x)\bigg(1+\frac{\sin(2\rho_t(x))}{2\rho_t(x)}\bigg)
  -\bigg(1-\frac{\sin(2\rho_t(x))}{2\rho_t(x)}\bigg)\bigg].
  \end{split}
  \end{equation}
Moreover, if $G$ satisfies \eqref{G-theta.1}, then as $\rho_t(x)\to 0$, the right hand side is equal to
$-\nu+o(\rho_t(x))$.
\end{lemma}

\begin{proof}
The last assertion is obvious.
In the sequel we concentrate on the proof of \eqref{sect-5-lem-3.1}.
It holds $\big|\nabla_{T_a} J_{\ell,k}(a)\big|^2=\rho_t^2(x)\big|\nabla_{e_a(t)} J_{\ell,k}(a)\big|^2
=\rho_t^2(x)|\dot J_{\ell,k}(a)|^2$.
By the expression of $J_{\ell,k}(a)$, we have
  \begin{align*}
  \dot J_{\ell,k}(a)&=-\frac{\<A_{\ell,k}(x_t), y_t\>+\<A_{\ell,k}(y_t), x_t\>}{\sin\rho_t(x)}e_a(t)
  +\rho_t(x)\bigg\{\frac{\cos(a\rho_t(x))}{\sin\rho_t(x)}\<A_{\ell,k}(y_t),N(t)\>\cr
  &\hskip63pt -\big[\sin(a\rho_t(x))+\cot\rho_t(x)\cos(a\rho_t(x))\big]\<A_{\ell,k}(x_t),N(t)\>\bigg\}N(t).
  \end{align*}
Thus
  \begin{align*}
  \big|\nabla_{T_a} J_{\ell,k}(a)\big|^2&=\frac{\rho_t^2(x)}{\sin^2\rho_t(x)}
  \big[\<A_{\ell,k}(x_t), y_t\>+\<A_{\ell,k}(y_t), x_t\>\big]^2
  +\rho_t^4(x)\frac{\cos^2(a\rho_t(x))}{\sin^2\rho_t(x)}\<A_{\ell,k}(y_t),N(t)\>^2\cr
  &\hskip13pt + \rho_t^4(x)\big[\sin(a\rho_t(x))+\cot\rho_t(x)\cos(a\rho_t(x))\big]^2\<A_{\ell,k}(x_t),N(t)\>^2\cr
  &\hskip13pt -2\rho_t^4(x)\frac{\cos(a\rho_t(x))}{\sin\rho_t(x)}
  \big[\sin(a\rho_t(x))+\cot\rho_t(x)\cos(a\rho_t(x))\big]\cr
  &\hskip25pt \times \<A_{\ell,k}(y_t),N(t)\>\<A_{\ell,k}(x_t),N(t)\>.
  \end{align*}
By Lemmas \ref{sect-2-lem-1} and \ref{sect-5-lem-2}, we arrive at
  \begin{equation}\label{sect-5-lem-3.2}
  \begin{split}
  \frac2{D_\ell}\sum_{k=1}^{D_\ell}\big|\nabla_{T_a} J_{\ell,k}(a)\big|^2
  &=2\rho_t^2(x)\big[1-\gamma_\ell(\cos\rho_t(x))\big]
  +\rho_t^4(x)\frac{\cos^2(a\rho_t(x))}{\sin^2\rho_t(x)}\cr
  &\hskip13pt +\rho_t^4(x)\big[\sin(a\rho_t(x))+\cot\rho_t(x)\cos(a\rho_t(x))\big]^2\cr
  &\hskip13pt -2\rho_t^4(x)\frac{\cos(a\rho_t(x))}{\sin\rho_t(x)} \big[\sin(a\rho_t(x))+\cot\rho_t(x)\cos(a\rho_t(x))\big]\cr
  &\hskip25pt \times \big[\cos\rho_t(x)\, \gamma_\ell(\cos\rho_t(x))-\sin^2\rho_t(x)\, \gamma'_\ell(\cos\rho_t(x))\big].
  \end{split}
  \end{equation}

Now we turn to the term $\big\<R(T_a,J_{\ell,k}(a))J_{\ell,k}(a),T_a\big\>$. Since
$T_a=\rho_t(x)e_a(t)$ and $J_{\ell,k}(a)=J_{\ell,k}^{(1)}(a)e_a(t)+J_{\ell,k}^{(2)}(a)N(t)$, we have
  \begin{align*}
  \big\<R(T_a,J_{\ell,k}(a))J_{\ell,k}(a),T_a\big\>&=|T_a|^2|J_{\ell,k}(a)|^2-\<T_a,J_{\ell,k}(a)\>^2\cr
  &=\rho_t^2(x)\big[J_{\ell,k}^{(1)}(a)^2+J_{\ell,k}^{(2)}(a)^2\big]-\rho_t^2(x)J_{\ell,k}^{(1)}(a)^2\cr
  &=\rho_t^2(x) J_{\ell,k}^{(2)}(a)^2.
  \end{align*}
Substituting the expression of $J_{\ell,k}^{(2)}(a)$ into the above equality yields
  \begin{align*}
  &\hskip13pt \big\<R(T_a,J_{\ell,k}(a))J_{\ell,k}(a),T_a\big\>\cr
  &=\rho_t^2(x)\bigg\{ \frac{\sin^2(a\rho_t(x))}{\sin^2\rho_t(x)} \<A_{\ell,k}(y_t),N(t)\>^2
  + \big[\cos(a\rho_t(x))-\cot\rho_t(x)\sin(a\rho_t(x))\big]^2\<A_{\ell,k}(x_t),N(t)\>^2\cr
  &\hskip46pt +2\frac{\sin(a\rho_t(x))}{\sin\rho_t(x)} \big[\cos(a\rho_t(x))-\cot\rho_t(x)\sin(a\rho_t(x))\big] \<A_{\ell,k}(y_t),N(t)\>\<A_{\ell,k}(x_t),N(t)\>\bigg\}.
  \end{align*}
Consequently, by Lemma \ref{sect-5-lem-2},
  \begin{align*}
  \frac2{D_\ell}\sum_{k=1}^{D_\ell}\big\<R(T_a,J_{\ell,k}(a))J_{\ell,k}(a),T_a\big\>
  &=\rho_t^2(x)\bigg\{\frac{\sin^2(a\rho_t(x))}{\sin^2\rho_t(x)}+ \big[\cos(a\rho_t(x))-\cot\rho_t(x)\sin(a\rho_t(x))\big]^2\cr
  &\hskip23pt +2\frac{\sin(a\rho_t(x))}{\sin\rho_t(x)} \big[\cos(a\rho_t(x))-\cot\rho_t(x)\sin(a\rho_t(x))\big]\cr
  &\hskip35pt\times \big[\cos\rho_t(x)\, \gamma_\ell(\cos\rho_t(x))-\sin^2\rho_t(x)\, \gamma'_\ell(\cos\rho_t(x))\big]\bigg\}.
  \end{align*}

Combining the above equality with \eqref{sect-5-lem-3.2}, we arrive at
  \begin{align*}
  &\hskip13pt \frac1{2\rho_t^2(x)}\cdot \frac2{D_\ell}\sum_{k=1}^{D_\ell}\big(\big|\nabla_{T_a} J_{\ell,k}(a)\big|^2
  - \big\<R(T_a,J_{\ell,k}(a))J_{\ell,k}(a),T_a\big\>\big)\cr
  &=[1-\gamma_\ell(\cos\rho_t(x))]+\frac1{2\sin^2\rho_t(x)}\big[\rho_t^2(x)\cos^2(a\rho_t(x))-\sin^2(a\rho_t(x))\big]\cr
  &\hskip13pt +\frac12 \big\{\big[\rho_t^2(x)\sin^2(a\rho_t(x))-\cos^2(a\rho_t(x))\big]
  +\cot^2\rho_t(x)\big[\rho_t^2(x)\cos^2(a\rho_t(x))-\sin^2(a\rho_t(x))\big]\cr
  &\hskip40pt +\cot\rho_t(x)\sin(2a\rho_t(x))[1+\rho_t^2(x)]\big\}\cr
  &\hskip13pt -\frac1{2\sin\rho_t(x)}\big[\cos\rho_t(x)\, \gamma_\ell(\cos\rho_t(x))-\sin^2\rho_t(x)\, \gamma'_\ell(\cos\rho_t(x))\big]\cr
  &\hskip25pt \times\big\{\sin(2a\rho_t(x))[1+\rho_t^2(x)]
  +2\cot\rho_t(x)\big[\rho_t^2(x)\cos^2(a\rho_t(x))-\sin^2(a\rho_t(x))\big]\big\}.
  \end{align*}
In the second term on the right hand side, we note that $1=\sin^2\rho_t(x)+\cos^2\rho_t(x)$ and obtain
  \begin{align*}
  &\hskip13pt \frac1{2\rho_t^2(x)}\cdot \frac2{D_\ell}\sum_{k=1}^{D_\ell}\big(\big|\nabla_{T_a} J_{\ell,k}(a)\big|^2
  - \big\<R(T_a,J_{\ell,k}(a))J_{\ell,k}(a),T_a\big\>\big)\cr
  &=[1-\gamma_\ell(\cos\rho_t(x))]+\frac12 [\rho_t^2(x)-1]
  +\cot^2\rho_t(x)\big[\rho_t^2(x)\cos^2(a\rho_t(x))-\sin^2(a\rho_t(x))\big]\cr
  &\hskip13pt +\frac12 \cot\rho_t(x)\sin(2a\rho_t(x))[1+\rho_t^2(x)]
  -\big[\cos\rho_t(x)\, \gamma_\ell(\cos\rho_t(x))-\sin^2\rho_t(x)\, \gamma'_\ell(\cos\rho_t(x))\big]\cr
  &\hskip45pt \times\bigg\{\frac{\sin(2a\rho_t(x))}{2\sin\rho_t(x)}[1+\rho_t^2(x)]
  +\frac{\cot(\rho_t(x))}{\sin\rho_t(x)}\big[\rho_t^2(x)\cos^2(a\rho_t(x))-\sin^2(a\rho_t(x))\big]\bigg\}.
  \end{align*}
Combining the terms with common factors yields
  \begin{align*}
  &\hskip13pt \frac1{2\rho_t^2(x)}\cdot \frac2{D_\ell}\sum_{k=1}^{D_\ell}\big(\big|\nabla_{T_a} J_{\ell,k}(a)\big|^2
  - \big\<R(T_a,J_{\ell,k}(a))J_{\ell,k}(a),T_a\big\>\big)\cr
  &=[1-\gamma_\ell(\cos\rho_t(x))]+\frac12 [\rho_t^2(x)-1]
  +\big[\rho_t^2(x)\cos^2(a\rho_t(x))-\sin^2(a\rho_t(x))\big]\cr
  &\hskip50pt \times \big\{\cot^2\rho_t(x)
  [1-\gamma_\ell(\cos\rho_t(x))]+ \cos\rho_t(x)\, \gamma'_\ell(\cos\rho_t(x))\big\}\cr
  &\hskip13pt +\frac12\sin(2a\rho_t(x))[1+\rho_t^2(x)]\big\{\cot\rho_t(x)[1-\gamma_\ell(\cos\rho_t(x))]
  + \sin\rho_t(x)\, \gamma'_\ell(\cos\rho_t(x))\big\}.
  \end{align*}
Rearranging the above equality leads to
  \begin{align*}
  &\hskip13pt \frac1{2\rho_t^2(x)}\cdot \frac2{D_\ell}\sum_{k=1}^{D_\ell}\big(\big|\nabla_{T_a} J_{\ell,k}(a)\big|^2
  - \big\<R(T_a,J_{\ell,k}(a))J_{\ell,k}(a),T_a\big\>\big)\cr
  &=[1-\gamma_\ell(\cos\rho_t(x))]\bigg\{1+\big[\rho_t^2(x)\cos^2(a\rho_t(x))-\sin^2(a\rho_t(x))\big]\cot^2\rho_t(x)\cr
  &\hskip30pt +\frac12\sin(2a\rho_t(x))[1+\rho_t^2(x)]\cot\rho_t(x)\bigg\}
  +\frac12 [\rho_t^2(x)-1] +\sin\rho_t(x)\, \gamma'_\ell(\cos\rho_t(x))\cr
  &\hskip30pt \times \bigg\{\frac12\sin(2a\rho_t(x))[1+\rho_t^2(x)]
  +\big[\rho_t^2(x)\cos^2(a\rho_t(x))-\sin^2(a\rho_t(x))\big]\cot\rho_t(x)\bigg\}.
  \end{align*}
Hence by the definition of $G(\theta)$,
  \begin{equation*}
  \hskip-15pt \begin{split}
  &\hskip13pt \frac1{2\rho_t^2(x)}\sum_{\ell=1}^\infty\frac{2b_\ell}{D_\ell}\sum_{k=1}^{D_\ell}
  \big(\big|\nabla_{T_a} J_{\ell,k}(a)\big|^2- \big\<R(T_a,J_{\ell,k}(a))J_{\ell,k}(a),T_a\big\>\big)\cr
  &=[G(0)-G(\rho_t(x))]\bigg\{1+\big[\rho_t^2(x)\cos^2(a\rho_t(x))-\sin^2(a\rho_t(x))\big]\cot^2\rho_t(x)\cr
  &\hskip30pt +\frac12\sin(2a\rho_t(x))[1+\rho_t^2(x)]\cot\rho_t(x)\bigg\}
  +\frac{G(0)}2 [\rho_t^2(x)-1]\cr
  &\hskip13pt -G'(\rho_t(x))\bigg\{\frac12\sin(2a\rho_t(x))[1+\rho_t^2(x)]
  +\big[\rho_t^2(x)\cos^2(a\rho_t(x))-\sin^2(a\rho_t(x))\big]\cot\rho_t(x)\bigg\}.
  \end{split}
  \end{equation*}
Finally, noticing that
  $$\int_0^1 \cos^2(a\rho_t(x))\,\d a=\frac12\bigg(1+\frac{\sin(2\rho_t(x))}{2\rho_t(x)}\bigg),\quad
  \int_0^1 \sin^2(a\rho_t(x))\,\d a=\frac12\bigg(1-\frac{\sin(2\rho_t(x))}{2\rho_t(x)}\bigg)$$
and
  $$\int_0^1 \sin(2 a\rho_t(x))\,\d a=\frac{1-\cos(2\rho_t(x))}{2\rho_t(x)},$$
thus, integrating both sides of the above equality from 0 to 1 gives us the desired equality.
\end{proof}

Now we can draw some conclusions from the above computations.

\begin{corollary}\label{sect-5-cor}
Assume that the function $G(\theta)$ defined in \eqref{G-theta} satisfies \eqref{G-theta.1}. Then
\begin{itemize}
\item[\rm(1)] the quadratic variation of $\D e(t)$ is dominated by $C_1 t$ for some $C_1>0$;
\item[\rm(2)] as $\rho_t(x)\to 0$, the quantity $\<\D e(t),e(t)\>$ vanishes.
\end{itemize}
\end{corollary}

\begin{proof}
The first assertion (1) is a consequence of the formula \eqref{sect-5-term-1.2},
while the second one follows from equality \eqref{sect-5-term-3} and Lemma \ref{sect-5-lem-3}.
\end{proof}

It is a pity that we are unable to give such a good characterisation of $e(t)$
as in \cite[Proposition 7.2]{AC12b}, where the underlying space is the two dimensional
flat torus instead of the sphere $S^2$. We finish the paper with the following remark.

\begin{remark}{\rm
Let $\{ b_\ell\}_{\ell\geq 1}$ be given as in \eqref{sect-2.1} with $\alpha\in(0,2)$.
Then by \cite[Lemma 9.5]{LeJan}, the function $G(\theta)$ defined in \eqref{G-theta} is differentiable
on $(0,\pi)$ and
  $$\lim_{\theta\to 0+}\frac{G(0)-G(\theta)}{\theta^\alpha}=KG(0)\quad \mbox{and}\quad
  \lim_{\theta\to 0+}\frac{G'(\theta)}{\theta^{\alpha-1}}=-\alpha KG(0)$$
for some constant $K>0$. Using these two limits, it is easy to show that as $\rho_t(x)\to 0$,
the right hand side of \eqref{sect-5-lem-3.1} is equal to $-\nu+O(\rho_t^\alpha(x))$,
hence the second assertion of Corollary \ref{sect-5-cor} still holds.
Moreover, we deduce from the equality \eqref{sect-5-term-1.2} that
  $$\frac{\d\<\xi,\xi\>_t}{\d t}\sim 2\nu +4\nu K(1+\alpha)\rho_t(x)^{\alpha-2}
  \quad \mbox{as }\rho_t(x)\to 0,$$
which means that the derivative of the quadratic variation of $\d\xi_t=
\frac1{\rho_t(x)}\dot J_0\big(\d_m x_t^N,\d_m y_t^N\big)$ tends to infinity. Thus we observe
the similar irregular behavior as in \cite[Proposition 7.2]{AC12b}. However, in this case,
the flow $g_t$ (and $\tilde g_t$) is not a flow of maps when the drift $u$ vanishes, as
shown in \cite[Theorem 9.4(c)]{LeJan} (notice that $\eta=1$; see also the first paragraph on p.858).}
\end{remark}

\medskip

\textbf{Acknowledgements.} The author would like to thank Professor Shizan Fang for
suggesting him to study the stochastic Lagrangian flows on the sphere, as an example
of the general framework of M. Arnaudon and A.B. Cruzeiro \cite{AC12b, AAC13}. He is
also grateful to the anonymous referees for their valuable suggestions which help to
correct the mistakes in the first version.

\end{document}